\documentclass{gtmon_a}
\pdfoutput=1

\usepackage{amssymb,amsmath}
\usepackage[all]{xy}

\proceedingstitle{Exotic homology manifolds (Oberwolfach 2003)}
\conferencestart{29 June 2003}
\conferenceend{5 July 2003}
\conferencename{Workshop on Exotic Homology Manifolds}
\conferencelocation{Oberwolfach Mathematics Institute, Oberwolfach,
Germany}

\editor{Frank Quinn}
\givenname{Frank}
\surname{Quinn}

\editor{Andrew Ranicki}
\givenname{Andrew}
\surname{Ranicki}

\title{Homologically arc-homogeneous ENRs}

\author{J\,L Bryant}
\givenname{J\,L}
\surname{Bryant}
\address{3005 Brandemere Drive\\Tallahassee\\Florida 32312\\USA}
\email{jbryant@math.fsu.edu}
\urladdr{}

\volumenumber{9}
\issuenumber{}
\publicationyear{2006}
\papernumber{1}
\startpage{1}
\endpage{6}

\doi{}
\MR{}
\Zbl{}


\received{23 March 2004}
\revised{}
\accepted{23 March 2004}
\published{22 April 2006}
\publishedonline{22 April 2006}
\proposed{}
\seconded{}
\corresponding{}
\version{}

\arxivreference{}  

\makeatletter
\def\cnewtheorem#1[#2]#3{\newtheorem{#1}{#3}[section]
\expandafter\let\csname c@#1\endcsname\c@theorem}

\let\xysavmatrix\xymatrix
\def\xymatrix{\disablesubscriptcorrection\xysavmatrix}
\let\wwcheck\check

\newtheorem{theorem}{Theorem}[section]
\cnewtheorem{lemma}[theorem]{Lemma}
\cnewtheorem{corollary}[theorem]{Corollary}
\theoremstyle{definition}
\cnewtheorem{defn}[theorem]{Definition}
\newtheorem*{definitions}{Definitions}

\makeatother  

\newcommand{\im}{\operatorname{im}}
\newcommand{\intr}{\operatorname{int}}
\def\enr{\textrm{ENR}}

\begin{document}

\begin{abstract} 
We prove that an arc-homogeneous Euclidean neighborhood
retract is a homology manifold.
\end{abstract}

\maketitle

\section{Introduction}
\label{S-intro}
The so-called Modified Bing--Borsuk Conjecture, which grew out of a
question in \cite{rBB65}, asserts that a homogeneous Euclidean
neighborhood retract is a homology manifold.  At this mini-workshop on
exotic homology manifolds, Frank Quinn asked whether a space that
satisfies a similar property, which he calls \emph{homological
arc-homogeneity}, is a homology manifold.  The purpose of this note is
to show that the answer to this question is yes.

\section{Statement and proof of the main result} \label{S-state}

\begin{theorem}\label{T-main}
Suppose that $X$ is an $n$--dimensional  homologically arc-homogeneous \enr.
Then $X$ is a homology $n$--manifold.
\end{theorem}

\begin{definitions}
A \emph{homology $n$--manifold} is a space $X$ having the property that,
for each $x\in X$,
$$H_k(X,X-x;{\mathbb{Z}}) \cong\begin{cases}
      \mathbb{Z}& k=n \\
      0& k\not= n.
\end{cases}$$
A \emph{Euclidean neighborhood retract} (\enr) is a space
homeomorphic to a closed subset of Euclidean space that is a
retract of some neighborhood of itself. A space $X$ is {\em
homologically arc-homogeneous}  provided that for every path
$\alpha\colon [0,1]\to X$, the inclusion induced map $$H_*(X\times
0, X\times 0 - (\alpha(0),0))\to H_*(X\times I,X\times I  -
\Gamma(\alpha))$$ is an isomorphism, where $\Gamma(\alpha)$
denotes the graph of $\alpha$.  The \emph{local homology sheaf}
$\mathcal{H}_k$ \emph{in dimension $k$} on a space $X$ is the sheaf with
stalks $H_k(X,X-x)$, $x\in X$.
\end{definitions}

By a result of Bredon \cite[Theorem~15.2]{gB67}, if an $n$--dimensional
space $X$ is cohomologically locally connected (over $\mathbb{Z}$),
has finitely generated local homology groups $H_k(X,X-x)$ for each
$k$, and if each $\mathcal{H}_k$ is locally constant, then $X$ is a
homology manifold.  We shall show that an $n$--dimensional, homologically
arc-connected  \enr\ satisfies the hypotheses of Bredon's theorem.

Assume from now on that $X$ represents an $n$--dimensional, homologically
arc-hom\-o\-gen\-e\-ous \enr.  Unless otherwise specified, all homology groups
are assumed to have integer coefficients. The following lemma is a
straightforward application of the definition and the Mayer--Vietoris
theorem.

\begin{lemma}\label{L-iso}
Given a path $\alpha\colon [0,1]\to X$ and $t\in [0,1]$, the inclusion
induced map $$H_*(X\times t, X\times t - (\alpha(t),t))\to H_*(X\times
I,X\times I  - \Gamma(\alpha))$$ is an isomorphism.
\end{lemma}

Given points $x,y\in X$, an arc $\alpha\colon I\to X$ from $x$ to $y$,
and an integer $k\ge 0$, let  $\alpha_*\colon H_k(X,X-x) \to H_k(X,X-y)$
be defined by the composition
$$\xymatrix{H_k(X,X-x)\ar[r]^-{\times 0}&H_*(X\times I,X\times I  -
\Gamma(\alpha))&H_k(X,X-y)\ar[l]_-{\times 1}.}$$
Clearly $(\alpha^{-1})_* = \alpha_*^{-1}$ and $(\alpha\beta)_* =
\beta_*\alpha_*$, whenever $\alpha\beta$ is defined.

\begin{lemma}\label{L-locuniq}
Given $x\in X$ and $\eta\in H_k(X,X-x)$, there is a neighborhood $U$
of $x$ in $X$ such that if
$\alpha$ and $\beta$ are paths in $U$ from $x$ to $y$,  then
$\alpha_*(\eta)=\beta_*(\eta)\in H_k(X,X-y)$.
\end{lemma}

\begin{proof}
We will prove the equivalent statement:  for each $x\in X$ and $\eta\in
H_k(X,X-x)$ there is a neighborhood $U$ of $x$ with $\alpha_*(\eta)= \eta$
for any loop $\alpha$ in $U$ based at $x$.

Suppose $x\in X$ and $\eta\in H_k(X,X-x)$.  Since $H_k(X,X-x)$ is the
direct limit of the groups $H_k(X,X-W)$, where $W$ ranges over the
(open) neighborhoods of $x$ in $X$, there is a neighborhood $U$ of $x$
and an $\eta_{\scriptscriptstyle U}\in H_k(X,X-U)$ that goes to $\eta$
under the inclusion $H_k(X,X-U)\to H_k(X,X-x)$.

Suppose $\alpha$ is a loop in $U$ based at $x$.
Let $\eta_{\alpha}\in H_k(X\times I,X\times I-\Gamma(\alpha))$ correspond
to $\eta$ under the isomorphism $\xymatrix{H_k(X,X-x)\ar[r]^-{\times
0}& H_k(X\times I,X\times I-\Gamma(\alpha))}$ guaranteed by homological
arc-homogeneity.

Let
$$\eta_{\scriptscriptstyle U\times I}=\eta_{\scriptscriptstyle
U}\times 0\in H_k(X\times I,X\times I-U\times I).$$
Then the image
of $\eta_{\scriptscriptstyle U\times I}$ in
$H_k(X\times I,X\times I-\Gamma(\alpha))$
is $\eta_{\alpha}$, as can be seen by chasing the
following diagram around the lower square.
\[\xymatrix{H_k(X,X-U)\ar[r]\ar[d]_{\cong}^{\times 1} &
H_k(X,X-x)\ar[d]^{\times 1}_{\cong}\\
        H_k(X\times I,X\times I-U\times I)\ar[r] & H_k(X\times I,X\times
        I-\Gamma(\alpha))\\
        H_k(X,X-U)\ar[r]\ar[u]^{\cong}_{\times 0} &
        H_k(X,X-x)\ar[u]_{\times 0}^{\cong}
        }\]
But from the upper square we see that $\eta_{\alpha}$ must also come
from $\eta$ after including into $X\times 1$.    That is, $\alpha_*(\eta)
= \eta$.  \end{proof}

\begin{corollary}\label{C-locuniq-two}
Suppose the neighborhood $U$ above is path connected and $F$ is the cyclic
subgroup of $H_k(X,X-U)$ generated by $\eta_{\scriptscriptstyle U}$.
Then, for every $y\in U$, the inclusion $H_k(X,X-U)\to H_k(X,X-y)$ takes
$F$ one-to-one onto the subgroup $F_y$ generated by $\alpha_*(\eta)$,
where $\alpha$ is any path in $U$ from $x$ to $y$.
\end{corollary}

\begin{lemma}\label{L-uniq}
Suppose $x,y\in X$ and $\alpha$ and $\beta$ are path-homotopic paths
in $X$ from $x$ to $y$.  Then $\alpha_* = \beta_*\colon H_k(X,X-x)
\to H_k(X,X-y)$.
\end{lemma}

\proof
By a standard compactness argument it suffices to show that, for a given
path $\alpha$ from $x$ to $y$ and element $\eta\in  H_k(X,X-x)$,  there
is an $\epsilon>0$ such that $\alpha_*(\eta) = \beta_*(\eta)$ for any path
$\beta$ from $x$ to $y$ $\epsilon$--homotopic (rel $\{x,y\}$) to $\alpha$.

Given a path $\alpha$ from $x$ to $y$, $\eta\in  H_k(X,X-x)$, and
$t\in I$, let $U_t$ be a path-connected neighborhood of $\alpha(t)$
associated with $(\alpha_t)_*(\eta)\in  H_k(X,X-\alpha(t))$ given by
\fullref{L-locuniq}, where $\alpha_t$ is the path $\alpha\big|[0,t]$.
There is a subdivision
$$\{0 = t_0 < t_1 < \cdots < t_m = 1\}$$
of $I$ such that
$\alpha\bigl([t_{i-1},t_i]\bigr)\subseteq U_i$
for each $i = 1, \ldots, m$,
where $U_i = U_t$ for some $t$.  There is an $\epsilon>0$ so that if
$H\colon I\times I\to X$ is an $\epsilon$--path-homotopy from $\alpha$
to a path $\beta$, then $H([t_{i-1},t_i]\times I)\subseteq U_i$.

For each $i = 1, \ldots, m$, let $\alpha_i = \alpha\big|[t_{i-1},t_i]$ and
$\beta_i = \beta\big|[t_{i-1},t_i]$, and for $i = 0, \ldots, m$, let $\gamma_i =
H\big|t_i\times I$  and $\eta_i = (\alpha_{t_i})_*(\eta)$.
By \fullref{C-locuniq-two},
$$(\alpha_i)_*(\eta_{i-1}) =
  (\gamma_{i-1}\beta_i\gamma_i^{-1})_*(\eta_{i-1})=\eta_i$$
where $\eta_0 = \eta$.  Since $\gamma_0$ and $\gamma_n$ are
the constant paths, it follows easily that $$\alpha_*(\eta) =
(\alpha_n)_*\cdots(\alpha_1)_*(\eta) = (\beta_n)_*\cdots(\beta_1)_*(\eta)
= \beta_*(\eta).\eqno\qed$$

\begin{proof}[Proof of \fullref{T-main}]
As indicated at the beginning of this note, we need only show that the
hypotheses of \cite[Theorem~15.2]{gB67} are satisfied.

Since $X$ is an \enr, it is locally contractible, and hence cohomologically
locally connected over $\mathbb{Z}$.

Given $x \in X$, let $W$ be a path-connected neighborhood of $x$ such that
$W$ is contractible in $X$.  If $\alpha$ and $\beta$ are two paths in $W$
from $x$ to a point $y\in W$, then $\alpha$ and $\beta$ are path-homotopic
in $X$.  Hence, by \fullref{L-uniq},
$\alpha_*\colon H_k(X,X-x) \to H_k(X,X-y)$ is a well-defined
isomorphism that is independent of $\alpha$ for every $k\ge 0$.  Hence,
$\mathcal{H}_k\big|W$ is the constant sheaf, and so $\mathcal{H}_k$ is
locally constant.

Finally, we need to show that the local homology groups of $X$ are
finitely generated.  This can be seen by working with a mapping cylinder
neighborhood of $X$.
Assume $X$ is nicely embedded in ${\mathbb{R}}^{n+m}$, for some $m\ge 3$,
so that  $X$  has a mapping cylinder neighborhood
$N = C_{\phi}$ of a map $\phi\colon \partial N\to X$, with mapping
cylinder projection $\pi\colon N\to X$ (see \cite{rM76}). Given a subset
$A\subseteq X$, let $A^* = \pi^{-1}(A)$ and
$\dot A = \phi^{-1}(A)$.

\begin{lemma}\label{L-dual}
If $A$ is a closed subset of $X$, then $H_k(X,X-A)\cong \wwcheck
H_c^{n+m-k}(A^*,\dot A).$
\end{lemma}

\begin{proof}
Suppose $A$ is closed in $X$. Since $\pi\colon N\to
X$ is a proper homotopy equivalence, $$H_k(X,X-A)\cong H_k(N,N-A^*).$$
Since
$\partial N$ is collared in $N$,
$$H_k(N,N-A^*)\cong H_k(\intr{N},\intr{N} - A^*),$$
and by Alexander duality,
$$H_k(\intr{N},\intr{N} - A^*)\cong\wwcheck
H_c^{n+m-k}(A^*-\dot A)$$ $$\cong
\wwcheck H_c^{n+m-k}(A^*,\dot A)$$
(since $\dot A$ is also collared in $A^*$).
\end{proof}

Since $X$ is $n$--dimensional, we get the following corollary.

\begin{corollary}\label{C-vanish}
If $A$ is a closed subset of $X$, then $\wwcheck H_c^{q}(A^*,\dot A) =
0$, if $q<m$ or $q>n+m.$
\end{corollary}

Thus, the local homology sheaf $\mathcal{H}_k$ of $X$ is isomorphic
to the Leray sheaf $\mathcal{H}^{n+m-k}$ of the map $\pi\colon N\to X$
whose stalks are $\wwcheck H^{n+m-k}(x^*,\dot x)$.  For each $k\ge 0$, this
sheaf is also locally constant, so there is a path-connected neighborhood
$U$ of $x$ such that $\mathcal{H}^{q}\big|U$ is constant for all $q\ge 0$.
Given such a $U$, there is a path-connected neighborhood $V$ of $x$ lying
in $U$ such that the inclusion of $V$ into $U$ is null-homotopic.  Thus,
for any coefficient group $G$, the inclusion $H^p(U,G)\to H^p(V,G)$
is zero if $p\not= 0$ and is an isomorphism for $p = 0$.

The Leray spectral sequences of $\pi\big|\pi^{-1}(U)$ and $\pi\big|\pi^{-1}(V)$
have $E_2$ terms
$$E^{p,q}_2(U) \cong H^p(U;\mathcal{H}^q), \qquad E^{p,q}_2(V) \cong
H^p(V;\mathcal{H}^q)$$
and converge to
$$E^{p,q}_{\infty}(U) \subseteq H^{p+q}(U^*,\dot U;\mathbb{Z}), \qquad
E^{p,q}_{\infty}(V) \subseteq H^{p+q}(V^*,\dot V;\mathbb{Z}),$$
respectively (see \cite[Theorem~6.1]{gB67}).  Since the sheaf
$\mathcal{H}^q$ is constant on $U$ and $V$, $H^p(U;\mathcal{H}^q)$
and $H^p(V;\mathcal{H}^q)$ represent ordinary cohomology groups with
coefficients in $G_q \cong \wwcheck H^{q}(x^*,\dot x).$

By naturality, we have the commutative diagram
\[\xymatrix{E^{0,q}_2(U) \ar[r]\ar[d]^{\cong} & E^{2,q-1}_2(U)\ar[d]_0\\
        E^{0,q}_2(V) \ar[r] & E^{2,q-1}_2(V)}\]
which implies that the differential $d_2\colon E^{0,q}_2(V) \to
E^{2,q-1}_2(V)$ is the zero map.  Hence,
$$E^{0,q}_3(V) = \ker\bigl(E^{0,q}_2(V) \to
  E^{2,q-1}_2(V)\bigr)/\im\bigl(E^{-2,q+1}_2(V)
\to E^{0,q}_2(V)\bigr) = E^{0,q}_2(V),$$
and, similarly,  $E^{0,q}_3(V) = E^{0,q}_4(V) = \cdots =
E^{0,q}_{\infty}(V).$  Thus,
$$H^{q}(V^*,\dot V;\mathbb{Z})\supseteq E^{0,q}_{\infty}(V) \cong
E^{0,q}_2(V) \cong H^0(V;\mathcal{H}^q)\cong H^0(V;G_q)\cong G_q.$$
Applying the same argument to the inclusion $(x^*,\dot x)\subseteq
(V^*,\dot V)$ yields the commutative diagram
\[\xymatrix{E^{0,q}_2(V) \ar[r]\ar[d]^{\cong} & E^{2,q-1}_2(V)\ar[d]_0\\
        E^{0,q}_2(x) \ar[r] & E^{2,q-1}_2(x)}\]
which, in turn, gives
\[\xymatrix{G_q\cong
H^0(V;G_q)\ar[r]^{\cong}\ar[d]_{\cong}&H^q(V^*,\dot
V;\mathbb{Z})\ar[d]\\
        G_q\cong H^0(x;G_q)\ar[r]^{\cong}&H^q(x^*,\dot x;\mathbb{Z})\cong
        G_q}\]
from which it follows that the inclusion $H^q(V^*,\dot V;\mathbb{Z})\to
H^q(x^*,\dot x;\mathbb{Z}) \cong G_q$ is an isomorphism of $G_q$.  Since
$(x^*,\dot x)$ is a compact pair in the manifold pair $(V^*,\dot
V)$, it has a compact manifold pair neighborhood $(W,\partial W)$.
Since the inclusion $H^{q}(V^*,\dot V)\to \wwcheck H^{q}(x^*,\dot x)$
factors through $H^{q}(W,\partial W)$, its image is finitely generated
for each $q$.  Hence, $H_k(X,X-x)\cong\wwcheck H^{n+m-k}(x^*,\dot x)$
is finitely generated for each $k$.
\end{proof}

The following theorem, which may be of independent interest, emerges
from the proof of \fullref{T-main}.

\begin{theorem}
Suppose $X$ is an $n$--dimensional \enr\ whose local homology sheaf
$\mathcal{H}_k$ is locally constant for each $k\ge 0$.  Then $X$ is a
homology $n$--manifold.
\end{theorem}

\bibliographystyle{gtart}
\bibliography{link}

\end{document}